\documentclass[12pt]{article}
\usepackage[cp1251]{inputenc}
\usepackage[T1]{fontenc}
\usepackage{amsmath,amsthm}
\usepackage{amsfonts,amssymb,enumerate}
\usepackage{url,paralist}
\usepackage{mathtools}
\usepackage[colorlinks=true,urlcolor=blue,linkcolor=red,citecolor=magenta]{hyperref}
\usepackage{enumerate}
\usepackage{tikz}

\usetikzlibrary{arrows,snakes,backgrounds,calc,math}
\tikzstyle{map}=[->,semithick]
\tikzstyle{arc}=[bend left,->,semithick]
\tikzstyle{rinclusion}=[right hook->,semithick]
\tikzstyle{linclusion}=[left hook->,semithick]

\usepackage{dsfont}
\usepackage[left=1in,right=1in,top=1in,bottom=1in]{geometry}

\theoremstyle{plain}
\newtheorem{theorem}{Theorem}[section]
\newtheorem{lemma}[theorem]{Lemma}
\newtheorem{corollary}[theorem]{Corollary}

\theoremstyle{definition}
\newtheorem{definition}[theorem]{Definition}

\newtheorem{remark}[theorem]{Remark}
\newtheorem{problem}[theorem]{Problem}

\newcommand{\R}{\mathbb{R}}
\newcommand{\Z}{\mathbb{Z}}

\newcommand{\cB}{\mathcal{B}}

\newcommand{\cN}{\mathcal{N}}
\newcommand{\cR}{\mathcal{R}}

\newcommand{\gh}{\mathrm{GH}}
\newcommand{\h}{\mathrm{H}}

\newcommand{\vr}[2]{\mathrm{VR}(#1;#2)}
\newcommand{\cech}[2]{\mathrm{\check{C}}(#1;#2)}

\newcommand{\og}{\overline{g}}
\newcommand{\oh}{\overline{h}}

\newcommand{\tig}{\tilde{g}}
\newcommand{\tih}{\tilde{h}}

\newcommand{\rom}[1]{{\em #1}}

\renewcommand{\)}{\rom)}

\renewcommand{\:}{\colon}
\newcommand{\0}{\emptyset}

\newcommand{\sm}{\setminus}
\let\sss\ss
\renewcommand{\ss}{\subseteq}
\renewcommand{\sp}{\supseteq}
\newcommand{\x}{\times}

\renewcommand{\a}{\alpha}
\renewcommand{\b}{\beta}
\newcommand{\D}{\Delta}

\newcommand{\e}{\varepsilon}
\renewcommand{\i}{\iota}
\renewcommand{\l}{\lambda}
\renewcommand{\r}{\rho}
\newcommand{\s}{\sigma}
\renewcommand{\S}{\Sigma}

\newcommand{\bB}{{\bar B}}

\newcommand{\cP}{\mathcal{P}}

\newcommand{\conv}{\operatorname{conv}}
\newcommand{\diam}{\operatorname{diam}}
\newcommand{\dis}{\operatorname{dis}}

\newcommand{\VR}{\operatorname{VR}}

\newcommand{\hhC}{\check{C}}

\def\ig#1#2#3#4{\begin{figure}[!ht]\begin{center}%
\includegraphics[height=#2\textheight]{#1.pdf}\caption{#4}\label{#3}%
\end{center}\end{figure}}
\title{Gromov--Hausdorff distance and Jung constant of finite-dimensional normed spaces}
\author{Henry Adams, Semeon A.~Bogatyi, Florian Frick,\\ Daniil A.~Ilyukhin, Alexander O.~Ivanov, Ivan N.~Mikhailov,\\ Alexey A.~Tuzhilin, Anton A.~Vikhrov}

\begin{document}
\maketitle

\begin{abstract}
For a finite-dimensional normed space $V$ and a subset $X$ with finite Hausdorff distance from $V$, we prove that the Gromov--Hausdorff distance between $X$ and $V$ is at least the Hausdorff distance between $X$ and $V$, divided by twice the relative Jung constant of $V$. If $V$ furthermore satisfies a certain intersection property, we show a stronger result where the relative Jung constant can be replaced with its absolute version.

\end{abstract}

\tableofcontents

%%%%%%%%%%%%%%%%%%%%%%%%%%%%%%
\section{Introduction}
%%%%%%%%%%%%%%%%%%%%%%%%%%%%%%
The Gromov--Hausdorff distance measures how different one metric space is from another: if the spaces are isometric, the distance is zero~\cite{BBI}.
As a rule, computing specific values of this distance is quite difficult~\cite{GrigorevIvanovTuzhilin,JiTuzhilin,LimMemoliSmith,harrison2023quantitative,Martin,GH-BU-VR}.
If the spaces are subsets of the same metric space, then the Gromov--Hausdorff distance between them is bounded above by the Hausdorff distance.
An alternative definition of the Gromov--Hausdorff distance gives that it is one half the infimum of distortions of all correspondences between the spaces~\cite{BBI}, so an upper bound can be obtained by computing the distortion of some correspondence.
However, obtaining lower bounds is a much more delicate problem.
One approach is to pass to the corresponding ultrametric spaces~\cite{LimMemoliSmith}, but this often yields very coarse estimates.
In~\cite{HvsGH}, methods of algebraic topology were applied to obtain a lower bound for the Gromov--Hausdorff distance via the Hausdorff distance in the case of closed manifolds and their subsets.
In~\cite{AdamsMajhiManinVirkZava}, finite metric graphs and their subsets were considered.
It was proved that, under certain restrictions on the subset, the Gromov--Hausdorff distance from the subset to the whole graph equals the Hausdorff distance.
In~\cite{IvaMikhTuz}, this result was extended to infinite trees.

In the present paper, we consider arbitrary finite-dimensional normed spaces and their subsets at finite Hausdorff distance.
In~\cite{MikhTuz} it was shown that an arbitrary subset of a finite-dimensional Euclidean space lies at a finite Hausdorff distance from the ambient space, if and only if it lies at a finite Gromov--Hausdorff distance from it.
We enhance, and estimate from below the Gromov--Hausdorff distance from a subset to the ambient space in terms of the Hausdorff distance and Jung constant of the space in case of an arbitrary norm.
We start with a relatively short proof of a weak estimate using methods of functional analysis.
Then we generalize the approach from~\cite{HvsGH} and obtain a stronger estimate.
As a consequence we obtain: 

\begin{theorem}
\label{thm:sup-equal}
For any finite-dimensional normed space $V$ endowed with sup-norm, and any subset $X \subseteq V$ such that $d_\h(X,V)<\infty$, we have $d_\gh(X,V)=d_\h(X,V)$.
\end{theorem}

This will follow immediately from the more general Theorem~\ref{thm:gh-lower-jung-new} combined with Lemma~\ref{lem:MaxNorm}.\\

%that for a space with the $\max$-norm, the Gromov--Hausdorff and Hausdorff distances coincide.

Henry Adams was supported by the Simons Foundation Travel Support for Mathematicians.
Florian Frick was supported by NSF CAREER Grant DMS 2042428.
Ivan N. Mikhailov was supported by grant No.~25-8-2-11-1 of Theoretical Physics and Mathematics Foundation <<Basis>>.
The work of Alexey A.~Tuzhilin was supported by grant No.\ 25-21-00152 of the Russian Science Foundation, by National Key R\&D Program of China (Grant No.\ 2020YFE0204200), as well as by the Sino-Russian Mathematical Center at Peking University.
Part of the work by Alexey A.~Tuzhilin was done in Sino-Russian Math.\ Center, and he thanks the Math.\ Center for the invitation and the hospitality.

%%%%%%%%%%%%%%%%%%%%%%%%%%%%%%
\section{Preliminary material}
\markright{\thesection.~Preliminary material}
%%%%%%%%%%%%%%%%%%%%%%%%%%%%%%

Let $X$ be an arbitrary metric space.
The distance between points $x,y\in X$ will be denoted by $|xy|$.
Now let $x\in X$, and let $r>0$ and $s\ge 0$ be real numbers.
By $B(x;r)=\bigl\{y\in X:|xy|<r\bigr\}$ and $\bB(x;s)=\bigl\{y\in X:|xy|\le s\bigr\}$ we denote the \emph{open\/} and \emph{closed\/} balls with center $x$ and radii $r$ and $s$, respectively.
If $A$ and $B$ are nonempty subsets of $X$, we set $|xA|=|Ax|=\inf\bigl\{|xa|:a\in A\bigr\}$ and $|AB|=|BA|=\inf\bigl\{|ab|:a\in A,\,b\in B\bigr\}$.
Furthermore, we define the \emph{open $r$-neighborhood\/} and the \emph{closed $s$-neighborhood\/} of a set $A$ by setting respectively
$$
B(A;r)=\bigl\{x\in X:|xA|<r\bigr\}\ \ \text{and}\ \ \bB(A;s)=\bigl\{x\in X:|xA|\le s\bigr\}.
$$

%%%%%%%%%%%%%%%%%%%%%%%%%%%%%%
\subsection{Hausdorff distance}
%%%%%%%%%%%%%%%%%%%%%%%%%%%%%%

For details concerning Hausdorff and Gromov--Hausdorff distances see~\cite{BBI}.

\begin{definition}\label{dfn:H}
For nonempty subsets $A$ and $A'$ of a metric space $X$, the \emph{Hausdorff distance from $A$ to $A'$} is the quantity
\begin{multline*}
d_\h(A,A')=\max\bigl\{\sup_{a\in A}|aA'|,\,\sup_{a'\in A'}|Aa'|\bigr\}\\ =
\inf\bigl\{r>0:A\ss B(A';r)\ \text{and}\ B(A;r)\sp A'\bigr\}\\ =
\inf\bigl\{s\ge0:A\ss\bB(A';s)\ \text{and}\ \bB(A;s)\sp A'\bigr\}.
\end{multline*}
\end{definition}

%%%%%%%%%%%%%%%%%%%%%%%%%%%%%%
\subsection{Gromov--Hausdorff distance}
%%%%%%%%%%%%%%%%%%%%%%%%%%%%%%

If $X$ and $Y$ are isometric metric spaces, we denote this fact by $X\approx Y$.

\begin{definition}\label{dfn:GH}
The \emph{Gromov--Hausdorff distance\/} between nonempty metric spaces $X$ and $Y$ is the quantity
$$
d_\gh(X,Y)=\inf\bigl\{d_\h(X',Y')\:X',Y'\ss Z,\,X'\approx X,\,Y'\approx Y\bigr\},
$$
where the infimum is taken over all metric spaces $Z$ and all isometric embeddings of $X$ and~$Y$ into~$Z$.
\end{definition}

Definition~\ref{dfn:GH} is not well suited for concrete computations.
There is an equivalent definition in terms of correspondences.
We recall the necessary notions and results.

A \emph{relation\/} between sets $X$ and $Y$ is any subset of the Cartesian product $X\times Y$.

\begin{definition}
For any nonempty metric spaces $X$, $Y$ and a nonempty relation $\s\ss X\times Y$, the \emph{distortion $\dis\s$ of the relation $\s$} is the quantity
$$
\dis\s=\sup\Bigl\{\bigl||xx'|-|yy'|\bigr|\:(x,y),\,(x',y')\in\s\Bigr\}.
$$
\end{definition}

A relation $R\ss X\times Y$ is called a \emph{correspondence\/} if for every $x\in X$ there exists $y\in Y$, and for every $y\in Y$ there exists $x\in X$, such that $(x,y)\in R$.
The set of all correspondences between $X$ and $Y$ is denoted by $\cR(X,Y)$.

\begin{theorem}[\cite{BBI}]\label{thm:GH-metri-and-relations}
For any nonempty metric spaces $X$ and $Y$, we have
$$
d_\gh(X,Y)=\frac12\inf\bigl\{\dis R\:R\in\cR(X,Y)\bigr\}.
$$
\end{theorem}

%%%%%%%%%%%%%%%%%%%%%%%%%%%%%%
\subsection{Jung constant}
%%%%%%%%%%%%%%%%%%%%%%%%%%%%%%
Let $X$ be a non-empty bounded subset of a metric space $Z$.
By the \emph{circumradius\/} or \emph{Chebyshev radius $R(X)$ of $X$} we mean the infimal possible $r$ such that $X$ belongs to some closed ball $\bB(z;r)$ with $z\in Z$:
$$
R(X)=\inf\bigl\{r:\exists\, z\in Z,\,X\ss\bB(z;r)\bigr\}.
$$
Another way to describe this radius is
$$
R(X)=\inf_{z\in Z}\sup_{x\in X}|zx|.
$$
Given a metric space $Z$, let $\cB_{>0}(Z)$ denote the collection of all bounded subsets of $Z$ whose diameter is greater than~$0$, i.e., $\cB_{>0}(Z)$ consists of all bounded $X\ss Z$ containing at least two different points.
\textbf{In this section, we shall always assume that $\cB_{>0}(Z)\ne\0$, i.e., the space $Z$ contains at least two different points.}

Define the \emph{Jung constant $J(Z)$ of a metric space $Z$} as follows:
$$
J(Z)=\sup_{X\in\cB_{>0}(Z)}\frac{R(X)}{\diam X}.
$$
Thus, for any $X\in\cB_{>0}(Z)$, we have $R(X)\le J(Z)\,\diam X$.

Notice that, for any metric space $Z$, we have $1/2\le J(Z)\le1$.
The first inequality follows from the triangle inequality, and the second one from the fact that the ball of radius $\diam X$ centered at a point of $X$ contains all of $X$.

For a normed space $V$, due to its homogeneity, we give another equivalent definition of the Jung constant.
Denote by $\cP_1(V)$ the set of all subsets $X$ of $V$ such that $\diam X=1$.
Then
$$
J(V)=\sup_{X\in\cP_1(V)}R(X).
$$

Let us present a few well-known examples of the Jung constants:
\begin{itemize}
\item for any normed space $V$ with the $\sup$-norm, we have $J(V)=1/2$ (see~\cite{jung1910kleinsten});
\item for any one-distance space $Z$ with at least $2$ points, i.e., when all the distances between different points of $Z$ equal $d$, we have $J(Z)=1$ (evident);
\item for an $n$-dimensional inner-product space $V$, we have $J(V)=\sqrt{\frac{n}{2(n+1)}}$ (see~\cite{jung1910kleinsten});
\item for any $n$-dimensional normed space $V$, we have $J(V)\le n/(n+1)$ (see~\cite{Bohnenblust,Amir1985});
\item for $V=\R^n$ with the $\ell_1$-norm, we have $J(V)=n/(n+1)$ if and only if there exists an $(n+1)\x(n+1)$ Hadamard matrix, which is the case for $n=1,2$ and some multiples of $4$ (see~\cite{Dol'nikov}).
\end{itemize}

Now we define the \emph{Gulevich number $G(V)$} of a normed space $V$ as follows; see~\cite{Gulevich}.
For each~$X\in\cP_1(V)$, we consider its convex hull $\conv X$, and calculate the \emph{measure of nonconvexity\/} of the set $X$ as
$$
\l(X)\coloneqq\sup_{c\in\conv X}|c\,X|=\sup_{c\in\conv X}\inf_{x\in X}|cx|.
$$
After that we set
$$
G(V)=\sup_{X\in\cP_1(V)}\l(X).
$$
Thus, for any bounded subset $X$ of a normed space $V$ and any point $c\in \conv X$, we have $|c\,X|\le G(V)\,\diam X$.
In particular, if $X$ is finite or, more generally, compact, then for each $c\in\conv X$ there exists some $x\in X$ with $|cx|\le G(V)\,\diam X$.

Now we formulate a theorem proved by Gulevich in~\cite{Gulevich}.
For this, we need to modify the definition of the Jung constant as follows.
First we will generalize the definition of the Chebyshev radius in terms suggested in~\cite{Amir1985}.
Namely, let $X$ be a non-empty bounded subset of a metric space $Z$, and let $Y$ be an arbitrary non-empty subset of $Z$.
By the \emph{relative Chebyshev radius $R_Y(X)$} we mean $\inf_{y\in Y}\sup_{x\in X}|xy|$.
Note that $R_Z(X)$ is the standard Chebyshev radius.
In other words, the relative Chebyshev radius is obtained from the standard one by restricting the possible positions of the centers of balls circumscribed around the set $X$.

In~\cite{Gulevich}, for a normed space $V$, the \emph{relative Jung constant\/} was defined as
$$
J_s(V)=\sup_{X\in\cP_1(V)}\bigl\{R_C(X),\  \text{where $C=\conv X$}\bigr\}.
$$

\begin{remark}
Notice that in~\cite{Gulevich} the author called $J_s(V)$ simply the Jung constant and denoted it as $J(V)$.
However, in subsequent papers, the authors renamed this value as we did here.
\end{remark}

After that, Gulevich proved the following result.

\begin{theorem}[\cite{Gulevich}]\label{thm:Gulevich}
For any Banach space $V$, we have $G(V)=J_s(V)$.
\end{theorem}

\begin{remark}
Generally speaking, $J_s(V)\ne J(V)$.
Indeed, let $V$ be $\R^3$ with the $\max$-norm.
Denote by $e_1$, $e_2$, $e_3$ the standard basis of $\R^3$, and let $C=\conv\{e_1,e_2,e_3\}$.
Take an arbitrary point $P=(x_1,x_2,x_3)\in C$, then $x_1+x_2+x_3=1$, and thus there exists $i$ such that $x_i\le1/3$.
Hence, we have $|Pe_i|\ge2/3$.
But $\diam\{e_1,e_2,e_3\}=1$, therefore $J_s(V)\ge2/3>1/2=J(V)$.
\end{remark}

\begin{remark}
The fact that Gulevich used the notation $J(V)$ for the relative Jung constant has led some authors to make a mistake, see for example~\cite[Theorem 12.1]{AlimovTsarkov2019}.
\end{remark}

Note that for any bounded subset $X$ of a normed space, and $C=\conv X$, each ball $B$ contains $X$ if and only if it contains $C$, because each ball in a normed space is convex.
This immediately implies that $R_Y(X)=R_Y(C)$ for any non-empty $Y$, and thus we have
$$
J_s(V)=\sup\bigl\{R_C(C):\text{$C\in\cP_1(V)$ and $C$ is convex}\bigr\}.
$$

\begin{theorem}[V.L.~Klee~\cite{Klee1960}, A.L.~Garkavi~\cite{Garkavi1964}]
\label{thm:KleeGarkavi}
In a normed space $V$, the equality $J(V)=J_s(V)$ holds if and only if $V$ is an inner-product space, or the dimension of $V$ is at most $2$.
\end{theorem}

Let us show some results on concrete values and estimates of the relative Jung constant~$J_s(V)$.
\begin{itemize}
\item For $n$-dimensional inner-product spaces $V$, we have $J_s(V)=J(V)=\sqrt{\frac{n}{2(n+1)}}$ (see Theorem~\ref{thm:KleeGarkavi}).
\item For infinite-dimensional Hilbert space $V$, we have $J_s(V)=J(V)=1/\sqrt2$ (see \cite{Routledge} and Theorem~\ref{thm:KleeGarkavi}).
\item For any $n$-dimensional normed space $V$, we have $J_s(V)\le n/(n+1)$ (see \cite[Proposition~2.12]{Amir1985}).
\item For $n$-dimensional space $V$ with the $\max$-norm and $n\ge2$, we have $J_s(V)=(n-1)/n$ (S.V.~Berdyshev~\cite{Berdyshev1998}).
For $1\le n\le2$, we have $J_s(V)=J(V)=1/2$ (see~\cite{jung1910kleinsten} and Theorem~\ref{thm:KleeGarkavi}).
\end{itemize}

Now we continue to our main results.

%%%%%%%%%%%%%%%%%%%%%%%%%%%%%%
\section{Weak estimate of the Gromov--Hausdorff distance}
\markright{\thesection.~Weak estimate of the Gromov--Hausdorff distance}
%%%%%%%%%%%%%%%%%%%%%%%%%%%%%%

\begin{theorem}
\label{thm:weak}
If $\bigl(V,\|\cdot\|\bigr)$ is a finite-dimensional normed space and $X\ss V$ is non-empty, then
$$
d_\gh(V,X)\ge\frac{1}{2\,J_s(V)}d_\h(V,X).
$$
\end{theorem}

\begin{proof}
If $d_\gh(V,X)=\infty$ then the inequality holds.
Suppose now that $d_\gh(V,X)<\infty$.
Choose an arbitrary finite $C>2\cdot d_\gh(V,X)$, then there exists a correspondence $R\in\cR(V,X)$ with $\dis R<C$.
Let $f\:V\to X$ be a mapping such that $f\ss R$, where $f$ denotes the graph of $f$ as well.
Then $\dis f\le\dis R<C$.

Choose an arbitrary positive $\e<C$.
Denote by $n$ the dimension of $V$, and choose an arbitrary basis $e_1,\ldots,e_n$ such that $\|e_i\|<\e/n$.
Then for any $0\le\l_i\le 1$ we have $\bigl\|\sum_{i=1}^n\l_i\,e_i\bigr\|<\e$.
For $(m_1,\ldots,m_n)\in\Z^n$ we let
$$
K_{m_1,\ldots,m_n}=\bigl\{(m_1+\l_1)\,e_1+\cdots+(m_n+\l_n)\,e_n:0\le\l_1,\ldots,\l_n\le1\bigr\},
$$
giving $\diam K_{m_1,\ldots,m_n}<\e$ and $V=\cup K_{m_1,\ldots,m_n}$.
We triangulate each of the parallelepipeds $K_{m_1,\ldots,m_n}$, obtaining a triangulation $\S$ of the entire $V$.
Note that for each $\D\in\S$, we have $\diam\D<\e$.

Let $Y$ be the set of vertices of the triangulation $\S$.
Let $f^*\:V\to V$ be defined as follows: $f^*|_Y=f$, and for each $v\in V\sm Y$ we choose an arbitrary $\D\in\S$ such that $v\in\D$, and if $\{y_0,\ldots,y_n\}\ss Y$ are the vertices of $\D$, we write $v=\sum_{i=0}^n\l_i\,y_i$ and put $f^*(v)=\sum_{i=0}^n\l_i\,f(y_i)$.
Since $\S$ is a triangulation, the mapping $f^*$ is well-defined and continuous.
We will need the following three lemmas.

\begin{lemma}\label{lemma:Diameter}
For each $\D\in\S$ we have $\diam f^*(\D)<C+\e$.
\end{lemma}

\begin{proof}
Let $\{y_0,\ldots,y_n\}\ss Y$ be the set of vertices of $\D$.
Take arbitrary points $w,\,w'\in f^*(\D)$, then $w=\sum_{i=0}^n\l_if(y_i)$ and $w'=\sum_{i=0}^n\mu_if(y_i)$ for non-negative $\l_i$ and $\mu_i$ satisfying $\sum_{i=0}^n\l_i=\sum_{i=0}^n\mu_i=1$.
Thus
\begin{multline*}
|ww'|=\biggl\|\sum_{i=0}^n\l_if(y_i)-\sum_{j=0}^n\mu_jf(y_j)\biggr\|=
\biggl\|\Bigl(\sum_{j=0}^n\mu_j\Bigr)\sum_{i=0}^n\l_if(y_i)-\Bigl(\sum_{i=0}^n\l_i\Bigr)\sum_{j=0}^n\mu_jf(y_j)\biggr\|\\
=\biggl\|\sum_{i,j}\l_i\mu_j\bigl(f(y_i)-f(y_j)\bigr)\biggr\|\le\sum_{i,j}\l_i\mu_j\bigl|f(y_i)f(y_j)\bigr|<\sum_{i,j}\l_i\mu_j(C+\e)=C+\e,
\end{multline*}
because $\dis f<C$ and $|y_iy_j|<\e$.
\end{proof}

\begin{lemma}\label{lemma:Extimate}
We have $\dis f^*\le 3C+4\e$.
\end{lemma}

\begin{proof}
Note that if $v\in\D\in\S$ and $\{y_0,\ldots,y_n\}\ss Y$ are the vertices of $\D$, then $v=\sum_{i=0}^n\l_i\,y_i$, where $0\le\l_i\le1$ for all $i$, and $\sum_{i=0}^n\l_i=1$.
Also, $|v\,y_j|<\e$ for all $j$, and we get
\begin{multline*}
\bigl|f^*(v)f^*(y_j)\bigr|=\biggl\|\sum_{i=0}^n\l_i\,f(y_i)-f(y_j)\biggr\|=
\biggl\|\sum_{i=0}^n\l_i\,\bigl(f(y_i)-f(y_j)\bigr)\biggr\|\\
\le\sum_{i=0}^n\l_i\,\bigl|f(y_i)f(y_j)\bigr|<\sum_{i=0}^n\l_i\,(C+|v\,y_j|)<C+\e,
\end{multline*}
thus $\bigl|f^*(v)f^*(y_j)\bigr|-|v\,y_j|<C+\e$.
Since $|v\,y_j|<\e<C$, we get $|v\,y_j|-\bigl|f^*(v)f^*(y_j)\bigr|<C+\e$, hence
$$
\Bigl|\bigl|f^*(v)f^*(y_j)\bigr|-|v\,y_j|\Bigr|<C+\e.
$$

Let $v_1,\,v_2\in V$ be arbitrary, and let $\D_1$ and $\D_2$ be simplices of $\S$ that contain $v_1$ and $v_2$, respectively.
Let $y_1,\,y_2\in Y$ be some vertices of $\D_1$ and $\D_2$, respectively.
Then
\begin{multline*}
\bigl|f^*(v_1)f^*(v_2)\bigr|-|v_1v_2|\\
\le\bigl|f^*(v_1)f^*(y_1)\bigr|+\bigl|f^*(y_1)f^*(y_2)\bigr|+\bigl|f^*(y_2)f^*(v_2)\bigr|
-|y_1y_2|+|y_1v_1|+|y_2v_2|\\
<(C+\e)+\Bigl(\bigl|f(y_1)f(y_2)\bigr|-|y_1y_2|\Bigr)+(C+\e)+
\e+\e<3C+4\e,
\end{multline*}
because $\bigl|f^*(v_1)f^*(y_1)\bigr|<C+\e$ and $\bigl|f^*(y_2)f^*(v_2)\bigr|<C+\e$ by Lemma~\ref{lemma:Diameter}.
Similarly,
\begin{multline*}
|v_1v_2|-\bigl|f^*(v_1)f^*(v_2)\bigr|\\
\le|v_1y_1|+|y_1y_2|+|y_2v_2|-\bigl|f^*(y_1)f^*(y_2)\bigr|+\bigl|f^*(v_1)f^*(y_1)\bigr|+\bigl|f^*(y_2)f^*(v_2)\bigr|\\
<\e+\Bigl(|y_1y_2|-\bigl|f(y_1)f(y_2)\bigr|\Bigr)+\e+(C+\e)+(C+\e)<3C+4\e,
\end{multline*}
thus $\dis f^*\le3C+4\e$.
\end{proof}

\begin{lemma}[{Webster~\cite[Proposition 4.1]{B75}}]\label{lemma:onto}
Let $f\:V\to W$ be a continuous map between finite-dimensional normed spaces such that $\dim V\ge\dim W$.
Suppose that $\dis f<\infty$.
Then $\dim V=\dim W$ and $f$ is surjective.
\end{lemma}

We return to the proof of Theorem~\ref{thm:weak}.
By Lemma~\ref{lemma:Extimate}, the mapping $f^*\colon V\to V$ has finite distortion.
This, together with Lemma~\ref{lemma:onto}, implies that the mapping $f^*$ is surjective, thus $V=\cup_{\D\in\S}f^*(\D)$.
Take an arbitrary $v\in V$, then there exists $\D\in\S$ such that $v\in f^*(\D)$.
Let $\{y_0,\ldots,y_n\}$ be the vertices of $\D$.
By definition of $f^*|_\D$, we have $v\in\conv\bigl\{f(y_0),\ldots,f(y_n)\bigr\}=f^*(\D)$.
By Lemma~\ref{lemma:Diameter}, we have $\diam f^*(\D)<C+\e$.
Note that $f^*(Y)=f(Y)\ss X$, therefore $\bigl\{f(y_0),\ldots,f(y_n)\bigr\}\ss X$.
By Theorem~\ref{thm:Gulevich}, the point $v$ is at distance at most $J_s(V)(C+\e)$ from a point in $\bigl\{f(y_0),\ldots,f(y_n)\bigr\}\ss X$.
Since $v$ was an arbitrary point of $V$, we get $d_\h(V,X)\le J_s(V)(C+\e)$, and due to arbitrariness of $\e$, we conclude that $d_\h(V,X)\le J_s(V)\,C$.
It remains to recall that $C$ was an arbitrary finite value with $C>2\cdot d_\gh(V,X)$, and thus we get $d_\h(V,X)\le2\,J_s(V)\,d_\gh(V,X)$.
\end{proof}

\begin{corollary}
For every subset~$X$ of a finite-dimensional normed space~$V$, the following conditions are equivalent: 
\begin{itemize}
    \item $X$ is an $\varepsilon$-net in~$V$ for some~$\varepsilon > 0$;
    \item $d_{H}(X, V) < \infty$;
    \item $d_{GH}(X, V) < \infty$.
\end{itemize}
\end{corollary}

\begin{corollary}
If $V$ is an $n$-dimensional normed space with the $\max$-norm, and $X\ss V$ is non-empty, then
\begin{itemize}
\item $d_\gh(V,X)=d_\h(V,X)$ if $n\le 2$;
\item $d_\gh(V,X)\ge \frac{n}{2(n-1)} d_\h(V,X)$ for $n\ge3$.
\end{itemize}
\end{corollary}

The estimate from Theorem~\ref{thm:weak} can be strengthened under certain additional assumptions.
A key role here is played by the condition we have termed the \emph{intersection property} (see the next section).

%%%%%%%%%%%%%%%%%%%%%%%%%%%%%%
\section{Intersection property}
\markright{\thesection.~Intersection property}
%%%%%%%%%%%%%%%%%%%%%%%%%%%%%%

Recall that by $B(x;r)$ and $\bB(x;r)$ we denote the open and the closed balls, respectively, with centers at $x$ and radii $r$.
The key point of the proof of the stronger estimate below uses the following property of normed spaces.

\emph{Let $V$ be a finite-dimensional normed space.
Given $\r>0$, there exists $t>0$ such that for each $t'>t$ the intersection of $V\sm\bB(0;t')$ with the intersection of any finite collection $\bigl\{B(x_i;r_i):r_i\le\r\bigr\}$ is either empty or contractible\/ \(in particular, connected\/\).} We call this restriction the \textbf{\emph{intersection property}}.

\begin{lemma}\label{lem:InterEuclid}
Each $n$-dimensional Euclidean space $V$ with $n<\infty$ satisfies the intersection property.
\end{lemma}

\begin{proof}
We represent the space $V$ as the coordinate hyperspace $x^{n+1}=0$ in the Euclidean space $\R^{n+1}$ with the Cartesian coordinates $x^1,\ldots,x^{n+1}$.
Fix $\rho>0$.
Let $S^n(R)\ss\R^{n+1}$ denote the sphere with center at the origin and radius $R$, and let $N=(0,\ldots,0,R)$.
Consider the stereographic projection $\nu\:S^n(R)\sm\{N\}\to V$.
Choose $R$ large enough so that each $B(x_i;r_i)$ has preimage a spherical ball contained in an open hemisphere.
Then choose $t$ large enough so that each $V\sm\bB(0;t')$ for $t'>t$ also has preimage a spherical ball contained in an open hemisphere.
Since such balls in $S^n(R)$ are geodesically convex, then their intersection, if not empty, is also geodesically convex and thus contractible.
\end{proof}

\begin{lemma}\label{lem:MaxNorm}
Each $n$-dimensional space $V$ with $n<\infty$, endowed with $\max$-norm, satisfies the intersection property.
\end{lemma}

\begin{proof}
Here we use the notations from the intersection property, and we choose $t$ such that $\r<t$.
To start with, let us note that the intersection $P\coloneqq \cap_iB(x_i;r_i)$ is a rectangular parallelepiped with hyperfaces parallel to the coordinate hyperplanes.
Since $r_i\le\rho<t$ for all $i$, the parallelepiped $P$ cannot intersect opposite hyperfaces of $\bB(0;t')$ for $t'>t$.

Let us put $W=V\sm\bB(0;t')$, and consider the case when $W\cap P\ne\0$.
If $P$ does not intersect $\bB(0;t')$, then $W\cap P=P$ is convex.
If $P$ intersects only one hyperface $F$ of $\bB(0;t')$, then $W\cap P=\Pi\cap P$, where $\Pi$ is the open half-space bounded by the hyperspace containing $F$ such that $\Pi\cap\bB(0;t')=\0$.
Thus $W\cap P$ is convex in this case.

Now, let $P$ intersect several hyperfaces $F_i$, $i\in I\ss\{1,\ldots,n\}$, of $\bB(0;t')$.
Then for each $i$ we consider the corresponding open half-space $\Pi_i$ whose boundary $\pi_i$ contains $F_i$, and such that $\Pi_i\cap\bB(0;t')=\0$ for all $i$.
Without loss of generality, we assume that $\Pi_i=\{(x^1,\ldots,x^n):x^i>t'\}$, where the $x^i$ are the Cartesian coordinates of $V$; we can always obtain this situation by the corresponding symmetry of $V$.
Since $P$ is open and intersects $\pi_i$, it intersects both half-spaces bounded by $\pi_i$, and in particular, it intersects $\Pi_i$.
Choose arbitrary $A_i\in P\cap\Pi_i$, and let $A_i=(x_i^1,\ldots,x_i^n)$.
Then by evident property of rectangular parallelepipeds with hyperfaces parallel to the corresponding coordinate subspaces, the point $A=(\max_ix_i^1,\ldots,\max_ix_i^n)$ belongs to $P$.
By definition of $\Pi_i$, we have $A\in\Pi_i$, thus $A$ belongs to the intersection of $P_i\coloneqq P\cap\Pi_i$.
Since each $P_i$ is convex, and $W\cap P=\cup P_i$, all of them, together with $W\cap P$, can be contracted to $A$.
\end{proof}

Let us generalize the previous result.
To do that, we use the following
\begin{lemma}[\cite{Ilyukhin}]
\label{lem:ContrInter}
If $V$ is a $2$-dimensional normed space, then $V$ satisfies the intersection property.
\end{lemma}

\begin{lemma}[\cite{Ilyukhin}]
\label{lem:CylNorms}
Let $B^2$ be a unit ball of an arbitrary norm in the plane $\R^2$, and let $B^n$ be Cartesian product of $B_2$ with the $(n-2)$-dimensional cube $K^{n-2}\ss\R^{n-2}$.
Then the norm in $\R^n$ with the unit ball $B^n$ satisfies the intersection property.
\end{lemma}

We call such norms \emph{cylindrical}.

\begin{remark}
Let us note that not all normed spaces satisfy the intersection property.
One can construct counterexamples in the normed $3$-space where the unit ball is the icosahedron, dodecahedron, or octahedron; see Figure~\ref{fig:IcosaDodecOcta}.

\ig{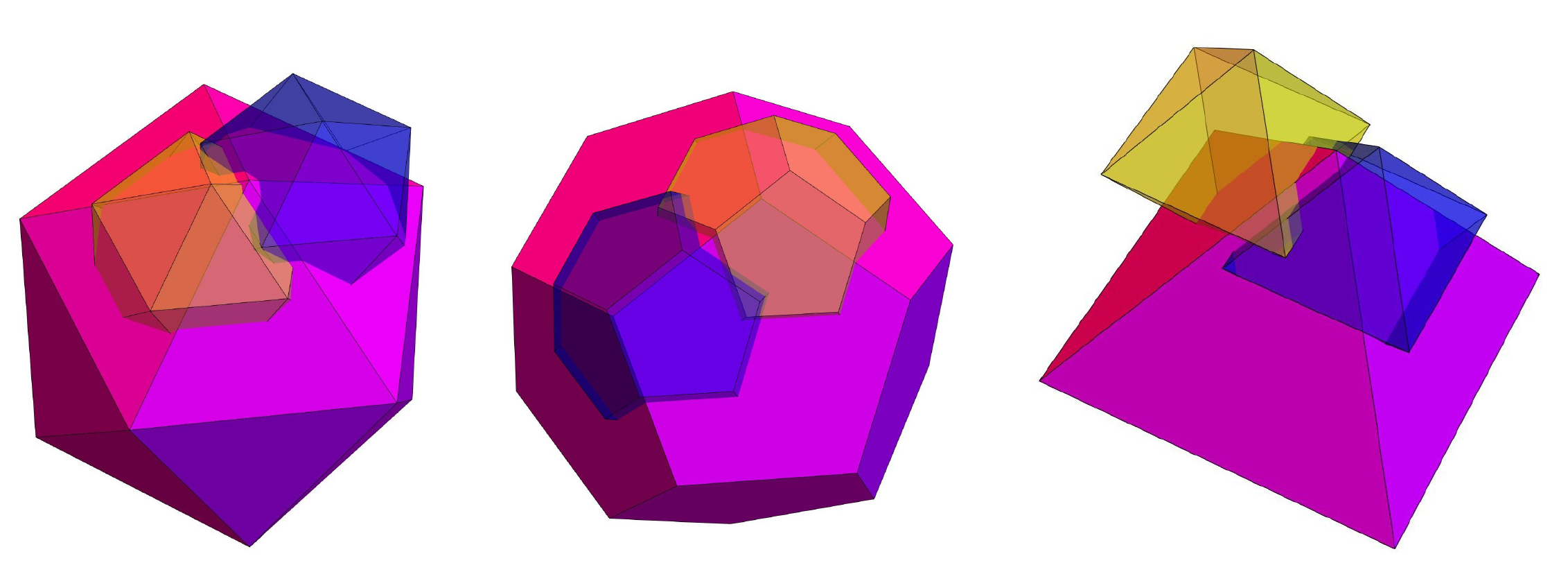}{0.27}{fig:IcosaDodecOcta}{The norms with unit ball the icosahedron, dodecahedron, or octahedron do not satisfy the intersection property.}

The situation is more intriguing than just these examples.
Let $x$, $y$, $z$ be the Cartesian coordinates of $\R^3$.
Consider the infinite square tube $T\coloneqq \max\bigl\{|x|,|y|\bigr\}\le a$ for some $a<1$.
Let $D^3$ be the standard Euclidean ball in $\R^3$, and define $B=T\cap D^3$.
Let $a$ be sufficiently small such that $B$ is a bounded part of the square tube $T$ with two spherical cups on its ends.
We create the norm whose unit ball is $B$.

Now we use the fact that for any square in the Euclidean plane, we can take an arbitrarily small disk and place it near a vertex of the square in such a way that the intersection of the disk with the complement of the square consists of two connected components.
This leads us to take a shrunk copy $B_1$ of $B$, shifted in such a way so that $(\R^3\sm B)\cap B_1$ is not convex.
This non-convexity enables us to take the second shrunk copy $B_2$ and obtain an intersection $(\R^3\sm B)\cap B_1\cap B_2$ that is not connected; see Figure~\ref{fig:NormedBall}.

\ig{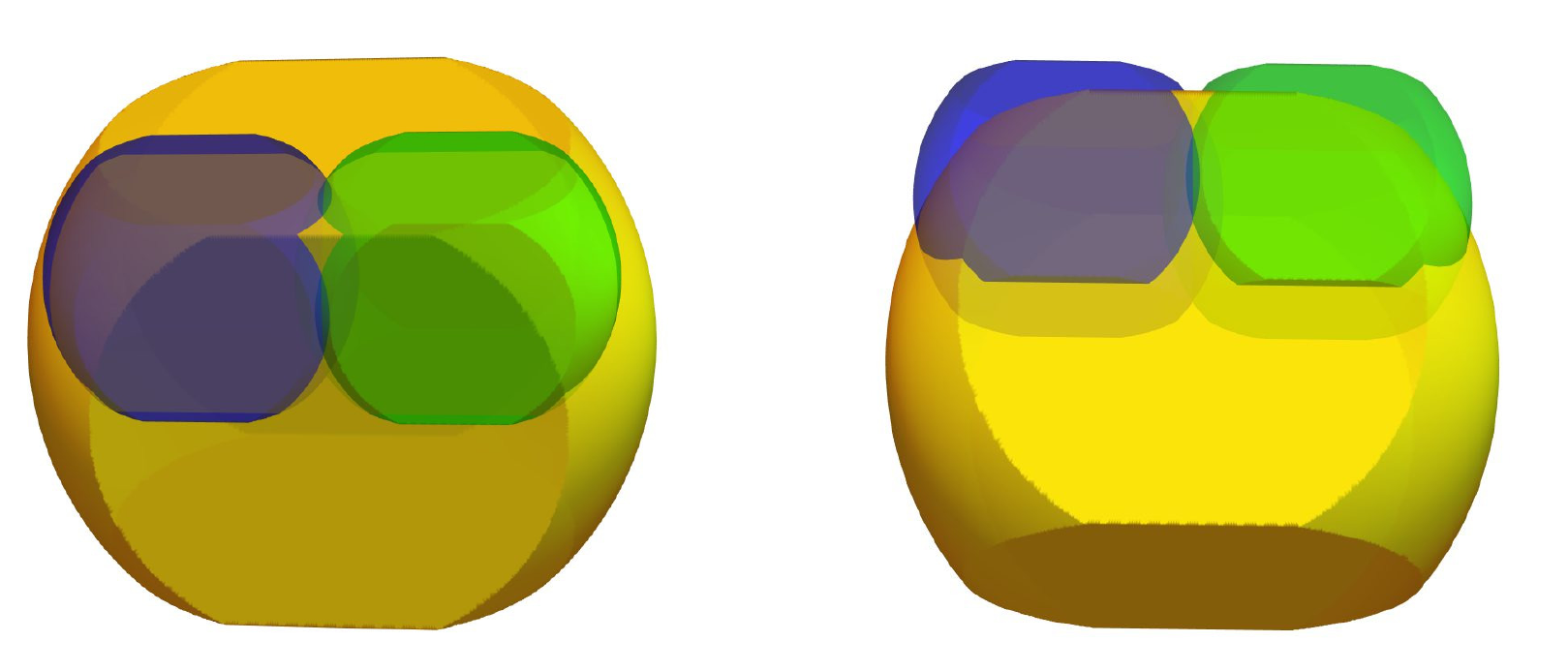}{0.27}{fig:NormedBall}{A norm which does not satisfy the intersection property, with $B$ in yellow, $B\cap B_1$ in blue, and $B\cap B_2$ in green.
The set $(\R^3\sm B)\cap B_1\cap B_2$ is not connected.
}

\end{remark}

\begin{problem}
Describe all finite-dimensional normed spaces satisfying the intersection property.
\end{problem}

%%%%%%%%%%%%%%%%%%%%%%%%%%%%%%
\section{Preliminaries towards a stronger estimate}\label{sec:complexes}
\markright{\thesection.~Preliminaries towards a stronger estimate}
%%%%%%%%%%%%%%%%%%%%%%%%%%%%%

In the next section, we will provide a stronger estimate of the Gromov--Hausdorff distance for normed spaces satisfying the intersection property.
First, we recall the necessary preliminary material here.

%%%%%%%%%%%%%%%%%%%%%%%%%%%%%%
\subsection{\v{C}ech complexes}
%%%%%%%%%%%%%%%%%%%%%%%%%%%%%%
Given a metric space $Z$, its subset $X$, and a real number $r>0$, the (ambient) \emph{\v{C}ech complex~$\hhC(X,r)$} is the simplicial complex with the vertex set $X$, such that $\s\coloneqq \{x_0,\ldots,x_k\}\ss X$ is a $k$-simplex of $\hhC(X,r)$ iff there exists $z\in Z$ for which $\{x_0,\ldots,x_k\}\ss B(z;r)$.

%%%%%%%%%%%%%%%%%%%%%%%%%%%%%%
\subsection{Coverings, nerves and \v{C}ech complexes}
%%%%%%%%%%%%%%%%%%%%%%%%%%%%%%
Given a topological space $Z$, a family $\{U_\a\}_{\a\in A}$ of its open subsets is said to be an \emph{open covering of $Z$} if $\cup_{\a\in A} U_\a=Z$.
A covering is said to be \emph{good} if any finite non-empty intersection~$\cap_{i=1}^kU_{\a_i}$ of its elements is contractible.
The \emph{Nerve complex~$\cN\bigl(\{U_\a\}\bigr)$ of the covering~$\{U_\a\}$} is the simplicial complex with the vertex set~$A$, such that $\s=\{\a_0,\ldots,\a_k\}\ss A$ is a $k$-simplex iff $\cap_{i=0}^k U_{\a_i}$ is not empty.
The following result is referred to as Leray's nerve lemma.

\begin{lemma}[{\cite[Corollary 4G.3]{Hatcher}}]
\label{lem:nerve}
Let $Z$ be a paracompact topological space, i.e., each of its open coverings has a locally finite refinement, and let $\{U_\a\}$ be a good open covering.
Then the nerve complex $\cN\bigl(\{U_\a\}\bigr)$ is homotopy equivalent to the space $Z$.
\end{lemma}

In particular, Lemma~\ref{lem:nerve} holds for any metric space $Z$, because each metric space is paracompact~\cite[Corollary 1, p. 979]{Stone}.

%%%%%%%%%%%%%%%%%%%%%%%%%%%%%%
\subsection{Vietoris--Rips complexes}
%%%%%%%%%%%%%%%%%%%%%%%%%%%%%%
Given a metric space $X$, the \emph{Vietoris--Rips complex $\VR(X,r)$} is the simplicial complex with the vertex set $X$, such that $\s\coloneqq \{x_0,\ldots,x_k\}\ss X$ is a $k$-simplex of $\VR(X,r)$ iff $\diam\s<r$.

The next lemma was proved in~\cite{ChazalDeSilvaOudot2014}; see also~\cite{HvsGH}.

\begin{lemma}\label{lem:correspondence-vr-general}
Let $X$ and $Y$ be non-empty metric spaces such that $r>2d_\gh(X,Y)$ for some positive real $r$.
Choose arbitrary mappings $h\:Y\to X$ and $g\:X\to Y$ such that $R\coloneqq h\cup g^{-1}\ss Y\x X$ is a correspondence with $\dis R<r$, where $g^{-1}$ is the relation inverse to $g$.
Then for any $\e>0$ and $\nu\ge0$, the mappings $h$ and $g$ induce the simplicial mappings $\oh$ and $\og$, respectively, 
\begin{equation*}
\begin{tikzpicture} [baseline=(current  bounding  box.center)]
\centering
\node (k2) at (-4.7,0) {$\vr{Y}{\e}$};
\node (k3) at (-1.6,0) {$\vr{X}{r+\e}$};
\node (k5) at (2.5,0) {$\vr{X}{r+\e+\nu}$};
\node (k6) at (6.8,0) {$\vr{Y}{2r+\e+\nu}$,};
\draw[map] (k2) to node[auto] {$\oh$} (k3);
\draw[map] (k5) to node[auto] {$\og$} (k6);
\draw[rinclusion] (k3)  to node[auto] {$\i_\nu$} (k5);
\end{tikzpicture}
\end{equation*}
where $\i_\nu$ is the inclusion, and the composition $\og\circ\i_\nu\circ\oh$ is contiguous to the inclusion $\vr{Y}{\e}\hookrightarrow\vr{Y}{2r+\e+\nu}$.
\end{lemma}

Recall that two simplicial maps are contiguous if for any simplex in the domain, its two images are contained in a common simplex of the codomain.
Contiguous simplicial maps induce homotopic maps on geometric realizations.

From the definition of the Jung constant, we immediately get the following lemma.

\begin{lemma}\label{lem:InclusionGeneral}
For any metric space $Z$, its non-empty subset $X$, and positive real $s$ and $r$,
\begin{itemize}
\item each simplex of $\VR(X;s)$ belongs to $\hhC\bigl(X;J(Z)s\bigr)$, and
\item each simplex of $\hhC(X;r)$ belongs to $\VR(X;2r)$.
\end{itemize}
Thus the identity map $Z\to Z$ generates the inclusions
$$
\i\:\VR(X;s)\to\hhC\bigl(X;J(Z)s\bigr)\ \ \text{and}\ \ \i'\:\hhC(X;r)\to\VR(X;2r),
$$
in particular,
$$
\VR(X;s)\xlongrightarrow{\i}\hhC\bigl(X;J(Z)s\bigr)\xlongrightarrow{\i'}\VR\bigl(X;2J(Z)s\bigr).
$$
\end{lemma}

%%%%%%%%%%%%%%%%%%%%%%%%%%%%%%
\section{Strong estimate of the Gromov--Hausdorff distance}
\markright{\thesection.~Strong estimate of the Gromov--Hausdorff distance}
%%%%%%%%%%%%%%%%%%%%%%%%%%%%%%

\begin{theorem}
\label{thm:gh-lower-jung-new}
Let $V$ be a finite-dimensional normed space, $\dim V\ge1$, $J\coloneqq J(V)$ the Jung constant of $V$, and $X\ss V$ such that $d_\h(X,V)<\infty$.
If $V$ satisfies the intersection property, then $d_\gh(X,V)\ge\dfrac1{2J}d_\h(X,V)$.
\end{theorem}

\begin{proof}
Suppose otherwise, i.e., $d_\gh(X,V)<\dfrac1{2J}d_\h(X,V)$.
Hence there exist positive $r$ and $\e$ such that
$$
2d_\gh(X,V)<r<r+\e<\dfrac1{J}d_\h(X,V)\ \ \text{and}\ \ (r+\e)J<2r.
$$
Since $r>2d_\gh(X,V)$, there exists a correspondence $C\in\mathcal{R}(V,X)$ with $\dis C<r$.
Let $h\:V\to X$ and $g\:X\to V$ be mappings such that $h,g^{-1}\ss C$, where $g^{-1}$ denotes the relation inverse to $g$.

By Lemmas~\ref{lem:InclusionGeneral} and~\ref{lem:correspondence-vr-general}, we get the following commutative diagram:
\begin{equation*}
\begin{tikzpicture} [baseline=(current  bounding  box.center)]
\centering
\node (k2) at (-4.7,0) {$\vr{V}{\e}$};
\node (k3) at (-1.6,0) {$\vr{X}{r+\e}$};
\node (k4) at (1,-1.5) {$\cech{X}{(r+\e)J}$};
\node (k5) at (3,0) {$\vr{X}{2(r+\e)J}$};
\node (k6) at (8,0) {$\vr{V}{r+2(r+\e)J}$,};
\draw[map] (k2) to node[auto] {$\oh'$} (k3);
\draw[rinclusion] (k3) to node[auto] {$\iota$} (k4);
\draw[rinclusion] (k4) to node[auto] {$\iota'$} (k5);
\draw[map] (k5) to node[auto] {$\og'$} (k6);
\draw[rinclusion, dashed] (k3) to (k5);
\end{tikzpicture}
\end{equation*}
where  $\i$, $\i'$, and the dashed arrow are inclusions (see Lemma~\ref{lem:InclusionGeneral}), the mappings $\oh'$, $\og'$ are generated by the mappings $h$ and $g$ (see Lemma~\ref{lem:correspondence-vr-general}), and the composition $\og'\circ\i'\circ\i\circ\oh'$ is contiguous to the inclusion map.

Lemma~\ref{lem:InclusionGeneral} and the inequality $J(V)\le1$ allow us to extend this diagram by two inclusions $\b'$ and $\b$:
\begin{equation*}
\begin{tikzpicture} [baseline=(current  bounding  box.center)]
\centering
\node (k1) at (-4.7,-1.5) {$\cech{V}{\e/2}$};
\node (k2) at (-4.7,0) {$\vr{V}{\e}$};
\node (k3) at (-1.6,0) {$\vr{X}{r+\e}$};
\node (k4) at (1,-1.5) {$\cech{X}{(r+\e)J}$};
\node (k5) at (3,0) {$\vr{X}{2(r+\e)J}$};
\node (k6) at (8,0) {$\vr{V}{r+2(r+\e)J}$};
\node (k7) at (8,-1.5) {$\cech{V}{r+2(r+\e)J}$};
\draw[rinclusion] (k1) to node[auto] {$\b'$} (k2);
\draw[map] (k2) to node[auto] {$\oh'$} (k3);
\draw[rinclusion] (k3) to node[auto] {$\iota$} (k4);
\draw[rinclusion] (k4) to node[auto] {$\iota'$} (k5);
\draw[map] (k5) to node[auto] {$\og'$} (k6);
\draw[rinclusion, dashed] (k3) to (k5);
\draw[rinclusion] (k6) to node[auto] {$\b$} (k7);
\end{tikzpicture}
\end{equation*}
Let us mention that here and in many places in what follows we can replace $\cech{V}{r+2(r+\e)J}$ with $\cech{V}{(r+2(r+\e)J)J}$; however for the latter one the last inequality in~\eqref{eq:strict} does not hold.
Note that again $\b\circ\og'\circ\i'\circ\i\circ\oh'\circ\b'$ is contiguous to the corresponding inclusion because both $\b'$ and $\b$ are inclusions.
Thus, we get the following short diagram:
\begin{equation}\label{eq:shortDiag}
\cech{V}{\e/2}\xlongrightarrow{\oh}\cech{X}{(r+\e)J}\xlongrightarrow{\og}\cech{V}{r+2(r+\e)J},
\end{equation}
where $\oh=\i\circ\oh'\circ\b'$ and $\og=\b\circ\og'\circ\i'$, and $\og\circ\oh$ is contiguous to the inclusion
$$
\i''\:\cech{V}{\e/2}\hookrightarrow\cech{V}{r+2(r+\e)J}.
$$
Since the mappings $\b$, $\b'$, $\i$, and $\i'$ are inclusions, the simplicial mappings $\oh$ and $\og$ are generated by the same $h$ and $g$, respectively.

Since $(r+\e)J<d_\h(X,V)$, the collection $\Bigl\{B\bigl(x;(r+\e)J\bigr)\Bigr\}_{x\in X}$ of open balls does not cover $V$.
Let $y\in V\sm\cup_{x\in X}B\bigl(x;(r+\e)J\bigr)$.
Changing if necessary the Cartesian coordinates in $V$ by a parallel translation, we assume w.l.o.g.\ that $y$ is the origin.
Moreover, the same procedure, being applied to those copies of $V$ that contain the images of $g$ and $h$, enables us to assume that $y=h(y)=g\bigl(h(y)\bigr)$, still with $\dis h\le\dis C<r$ and $\dis g\le\dis C<r$.

By definition, the \v{C}ech complexes from~(\ref{eq:shortDiag}) are the nerve complexes of the corresponding families of open balls in $V$, namely,
\begin{equation}\label{eq:nerves11}
\cN\bigl(\{B(x;\e/2)\}_{x\in V}\bigr)\xlongrightarrow{\oh}
\cN\bigl(\{B(x;(r+\e)J)\}_{x\in X}\bigr)\xlongrightarrow{\og}
\cN\bigl(\{B(x;r+2(r+\e)J)\}_{x\in V}\bigr).
\end{equation}

Recall that the mappings $\oh$ and $\og$ were induced by the mappings $h$ and $g$ with $\dis h<r$ and $\dis g<r$.

Since all balls are convex, their intersections are convex and therefore contractible.
Thus the families of open balls are good, and so the geometric realizations of their nerves are homotopy equivalent to the unions of balls forming the families (see the remark after Lemma~\ref{lem:nerve}).
In what follows, we wish to add complements of sufficiently big balls in such a way that the reconstructed families will be good as well.
However, the complement of a ball is not convex, and the intersection of the complement of a ball with a convex set can be non-contractible.
%(consider the standard disk in the Euclidean plane and intersect it with a small square such that its upper side tangent to the upper point of the disk, and then shift slightly the square downward to get disconnected intersection).
%However, if the convex sets are obtained as the intersection of some balls with respect to the same norm, the situation is much better.
This is what leads to the intersection property.

For $n=\dim V$, consider Euclidean space $\R^{n+1}$ with the standard Cartesian coordinates $x^1,\ldots,x^{n+1}$.
Denote by $\|\cdot\|$ a norm on the coordinate subspace $\{x^{n+1}=0\}\ss\R^{n+1}$ such that this subspace with this norm is linear isometric to $V$.
Thus, in what follows, we denote by $V$ the coordinate subspace $\{x^{n+1}=0\}$ endowed with this norm.
In addition, for any $x,x'\in V$, we use $|xx'|$ for the distance between these points with respect to the $\|\cdot\|$-norm.
Let $p\:S^n\setminus\{N\}\to V$ be the standard stereographic projection from the standard unit sphere $S^n\ss\R^{n+1}$ to $V$, where $N=(0,\ldots,0,1)$ is the north pole of $S^n$.

Let $t>0$ be large enough so that it satisfies the condition of the intersection property with $\r=r+2(r+\e)J$.
Then we put
$$
U_3\coloneqq \{x\in V:\|x\|>t\},\ \ U_2=\{x\in V:\|x\|>t+6r\},\ \ U_1=\{x\in V:\|x\|>t+9r\},
$$
and $V_i=p^{-1}(U_i)\cup\{N\}\ss S^n$.

We now show that~\eqref{eq:nerves11} extends into a diagram of \emph{larger\/} nerve complexes and simplicial maps
\begin{multline}\label{eq:nerves21}
\cN\bigl(\{B(x;\e/2)\}_{x\in V}\cup\{U_1\}\bigr)\xlongrightarrow{\oh}\\
\xlongrightarrow{\oh}\cN\bigl(\{B(x;(r+\e)J)\}_{x\in X}\cup\{U_2\}\bigr)\xlongrightarrow{\og}\\
\xlongrightarrow{\og}\cN\bigl(\{B(x;r+2(r+\e)J)\}_{x\in V}\cup\{U_3\}\bigr)
\end{multline}
in which we define $\oh(U_1)=U_2$ and $\og(U_2)=U_3$.

Let us check that this new map $\oh$ remains simplicial.
Choose an arbitrary simplex $\D=\{W_0,\ldots,W_k\}\in\cN\bigl(\{B(x;\e/2)\}_{x\in V}\cup\{U_1\}\bigr)$, and assume that for $i=0,\ldots,k-1$, the sets $W_i$ are open balls $B(x_i;\e/2)$ in $\bigl(V,\|\cdot\|\bigr)$ for some points $x_i\in V$.
If $W_k=B(x_k;\e/2)$, then $\D\in\cN\bigl(\{B(x;\e/2)\}_{x\in V}\bigr)$, hence
$$
\oh(\D)\in\cN\bigl(\{B(x;(r+\e)J)\}_{x\in X}\bigr)\ss\cN\bigl(\{B(x;(r+\e)J)\}_{x\in X}\cup\{U_2\}\bigr)
$$
because of~\eqref{eq:nerves11}.
If $W_k=U_1$, then $\bigl(\cap_{i=0}^{k-1}W_i\bigr)\cap U_1\ne\0$, and let $z$ be a point from this intersection.
Then
$$
\bigl\{h(x_0),\ldots,h(x_{k-1}),h(z)\bigr\}\in\vr{X}{r+\e}\ss\cech{X}{(r+\e)J},
$$
hence there exists $w\in V$ such that $\bigl\{h(x_0),\ldots,h(x_{k-1}),h(z)\bigr\}\ss B_{(r+\e)J}(w)$, i.e., $\bigl|h(x_i)w\bigr|<(r+\e)J$ for all $i=0,\ldots,k-1$, and $\bigl|h(z)w\bigr|<(r+\e)J$.
Since $(r+\e)J<2r$, we have
$$
|yw|\ge\bigl|yh(z)\bigr|-\bigl|h(z)w\bigr|\ge|yz|-r-\bigl|h(z)w\bigr|>t+9r-r-2r=t+6r,
$$
thus $w\in U_2$, therefore $w\in\Bigl(\cap_{i=0}^{k-1}B\bigl(h(x_i);(r+\e)J\bigr)\Bigr)\cap U_2\ne\0$,  and we get $\oh(\D)\in\cN\bigl(\{B(x;(r+\e)J)\}_{x\in X}\cup\{U_2\}\bigr)$.

Let us check that the second new map $\og$ also remains simplicial.
Choose an arbitrary simplex $\D=\{W_0,\ldots,W_k\}\in\cN\bigl(\{B(x;(r+\e)J)\}_{x\in X}\cup\{U_2\}\bigr)$, and assume that for $i=0,\ldots,k-1$, the sets $W_i$ are open balls $B\bigl(x_i;(r+\e)J\bigr)$ in $\bigl(V,\|\cdot\|\bigr)$ for some $x_i\in X$.
If $W_k=B(x_k;(r+\e)J)$, $x_k\in X$, then, as above, $\D\in\cN\bigl(\{B(x;(r+\e)J)\}_{x\in V}\bigr)$, hence
$$
\og(\D)\in\cN\bigl(\{B(x;r+2(r+\e)J)\}_{x\in V}\bigr)\ss\cN\bigl(\{B(x;r+2(r+\e)J)\}_{x\in V}\cup\{U_3\}\bigr)
$$
because of~\eqref{eq:nerves11}.
If $W_k=U_2$, then there exists $z\in\bigl(\cap_{i=0}^{k-1}B(x_i;(r+\e)J\bigr)\cap U_2$, and we have
$$
\bigl\{g(x_0),\ldots,g(x_{k-1}),g(z)\bigr\}\in\vr{X}{r+2(r+\e)J}\ss\cech{X}{r+2(r+\e)J}.
$$
Hence there exists $w\in V$ such that $\bigl\{g(x_0),\ldots,g(x_{k-1}),g(z)\bigr\}\ss B_{r+2(r+\e)J}(w)$, i.e., $\bigl|g(x_i)w\bigr|<r+2(r+\e)J$ for all $i=0,\ldots,k-1$, and $\bigl|g(z)w\bigr|<r+2(r+\e)J$.
Since $(r+\e)J<2r$, then $\bigl|g(z)w\bigr|<5r$, and we have
$$
|yw|\ge\bigl|yg(z)\bigr|-\bigl|g(z)w\bigr|\ge|yz|-r-\bigl|g(z)w\bigr|>t+6r-r-5r=t,
$$
thus $w\in U_3$, therefore $w\in\Bigl(\cap_{i=0}^{k-1}B\bigl(g(x_i);r+2(r+\e)J\bigr)\Bigr)\cap U_3\ne\0$,  and we get $\og(\D)\in\cN\bigl(\{B(x;r+2(r+\e)J)\}_{x\in V}\cup\{U_3\}\bigr)$.

Let us show that the composition $\og\circ\oh$ in~\eqref{eq:nerves21} is contiguous to the inclusion map
$$
\cN\bigl(\{B(x;\e/2)\}_{x\in V}\cup\{U_1\}\bigr) \to \cN\bigl(\{B(x;r+2(r+\e)J)\}_{x\in V}\cup\{U_3\}\bigr)
$$
which maps $U_1$ to $U_3$.
Take an arbitrary simplex $\D=\{W_0,\ldots,W_k\}\in\cN\bigl(\{B(x;\e/2)\}_{x\in V}\cup\{U_1\}\bigr)$.
If all $W_i$ are balls, then the proof is the same as the one of~\cite[Lemma~2.1]{HvsGH}, because $\b$ and $\b'$ are inclusions.

Suppose now that $W_i=B(x_i;\e/2)$, $i=0,\ldots,k-1$, and $W_k=U_1$.
Choose an arbitrary $z\in\cap_{i=0}^kW_i$, then $|zx_i|<\e/2$ and $z\in U_1$.
Put $y_i=g\bigl(h(x_i)\bigr)$.
Since $\bigl(z,h(z)\bigr),\,\bigl(y_i,h(x_i)\bigr)\in C$, we get
\begin{equation}\label{eq:strict}
|zy_i|<\bigl|h(z)h(x_i)\bigr|+r<|zx_i|+2r<\e/2+2r\le r+2(r+\e)J,
\end{equation}
because $J\ge1/2$ (for any metric space).
Moreover, $z\in U_1\ss U_3$, thus
$$
z\in\bigl(\cap_{i=0}^{k-1}B(y_i;r+2(r+\e)J)\bigr)\cap\bigl(\cap_{i=0}^{k-1}B(x_i;r+2(r+\e)J)\bigr)\cap U_1\cap U_3,
$$
therefore both $\D$ and $\og\bigl(\oh(\D)\bigr)$ are contained in a simplex from $\cN\bigl(\{B(x;r+2(r+\e)J)\}_{x\in V}\cup\{U_3\}\bigr)$, i.e., the mappings $\og\circ\oh$ and $\i$ are contiguous.

Recall that by $p\:S^n\setminus\{N\}\to V$ we denoted the stereographic projection, where $N$ is the north pole of $S^n$.
Let $V_s(x)=p^{-1}\bigl(B(x;s)\bigr)$.
Denote by $\tih$ and $\tig$ the mappings $p^{-1}\circ\oh\circ p$ and  $p^{-1}\circ\og\circ p$, respectively, extended to $N$ as $\tih(N)=\tig(N)=N$.
We get the simplicial chain
\begin{multline}\label{eq:nerves31}
\cN\bigl(\{V_{\e/2}(x)\}_{x\in V}\cup\{V_1\}\bigr)\xlongrightarrow{\tih}\\
\xlongrightarrow{\tih}\cN\bigl(\{V_{(r+\e)J}(x)\}_{x\in X}\cup\{V_2\}\bigr)\xlongrightarrow{\tig}\\
\xlongrightarrow{\tig}\cN\bigl(\{V_{r+2(r+\e)J}(x)\}_{x\in V}\cup\{V_3\}\bigr).
\end{multline}

As we noticed above, the balls in the nerves from~\eqref{eq:nerves21} form a good family.
Also, when adding $U_i$, we again obtain good families due to the choice of $t$ and $\r$ as in the intersection property.
Therefore the open sets in the nerves from~\eqref{eq:nerves31} form good families as well.
Hence, by Lemma~\ref{lem:nerve}, all of the complexes in~\eqref{eq:nerves31} are homotopy equivalent to the unions of the open sets forming them.
Thus the first complex and the last complex are homotopy equivalent to $S^n$, but the middle complex is homotopy equivalent to an open proper subset $U$ of $S^n$.
Therefore, the $n$th homologies of the first and the last complexes equal $H_n(S^n)\ne0$, and of the middle one equals $H_n(U)=0$.

However, we have the corresponding sequence of homomorphisms
$$
H_n(S^n)\xlongrightarrow{\tih_*}H_n(U)\xlongrightarrow{\tig_*}H_n(S^n),
$$
such that $\tig_*\circ\tih_*=(\tig\circ\tih)_*$, and since $H_n(U)=0$, we get $(\tig\circ\tih)_*=0$.
Note $\tig\circ\tih$ is contiguous with the inclusion, which, by the nerve lemma, is a homotopy equivalence.
Thus $(\tig\circ\tih)_*$ is an isomorphism on the non-zero group $H_n(S^n)$, a contradiction.

Therefore, 
% it must be that $r+\e\ge \dfrac1{J}d_\h(X,V)$ for all $r > 2d_\gh(X,V)$ and $\e>0$ sufficiently close to $0$.
% So, 
$d_\gh(X,V)\ge\dfrac1{2J}d_\h(X,V)$, as desired.
\end{proof}

Recall that Theorem~\ref{thm:sup-equal}
states that for a subset $X$ of a finite-dimensional normed space~$V$ endowed with the sup-norm we have that $d_\gh(X,V) = d_\h(X,V)$, provided the latter is finite.

\begin{proof}[Proof of Theorem~$\ref{thm:sup-equal}$]
Since $V$ has the sup-norm, it has the intersection property by Lemma~\ref{lem:MaxNorm}, and also $J(V)=1/2$.
Therefore Theorem~\ref{thm:gh-lower-jung-new} gives 
$d_\gh(X,V)\ge\dfrac1{2J}d_\h(X,V)=d_\h(X,V)$.
\end{proof}

% \begin{corollary}
% For any finite-dimensional normed space $V$ endowed with sup-norm, and any its subset $X$ such that $d_\h(X,V)<\infty$, we have $d_\gh(X,V)=d_\h(X,V)$.
% \end{corollary}

%\begin{proof}[Proof of Theorem~$\ref{thm:sup-equal}$]
%It remains to verify the inequality $d_\gh(X,V)\ge d_\h(X,V)$ that follows from Theorem~\ref{thm:gh-lower-jung-new} because $J(V)=1/2$.
%\end{proof}

\begin{corollary}
Let $V$ be a finite-dimensional normed space with $\dim V\ge1$ such that either
\begin{itemize}
\item the norm is cylindrical, or
\item $\dim V=2$.
\end{itemize}
If $J\coloneqq J(V)$ is the Jung constant of $V$, and $X\ss V$ such that $d_\h(X,V)<\infty$, then $d_\gh(X,V)\ge\dfrac1{2J}d_\h(X,V)$.
\end{corollary}

\let\ss\sss


\begin{thebibliography}{99}

\bibitem{GH-BU-VR}
H.~Adams, J.~Bush, N.~Clause, F.~Frick, M.~G\'{o}mez, M.~Harrison, R.A.~Jeffs, E.~Lagoda, S.~Lim, F.~M\'{e}moli, M.~Moy, N.~Sadovek, M.~Superdock, D.~Vargas, Q.~Wang, and L.~Zhou.
\newblock Gromov--{H}ausdorff distances, {B}orsuk--{U}lam theorems, and {V}ietoris--{R}ips complexes.
\newblock Accepted to appear in \emph{Algebraic \& Geometric Topology}, 2026.

\bibitem{HvsGH}
H.~Adams, F.~Frick, S.~Majhi, and N.~McBride.
\newblock Hausdorff vs {G}romov--{H}ausdorff distances.
\newblock {\em Discrete \& Computational Geometry}, 75:1217--1246, 2026.

\bibitem{AdamsMajhiManinVirkZava} H.~Adams, S.~Majhi, F.~Manin, Z.~Virk, and N.~Zava.
\newblock Lower-bounding the Gromov--Hausdorff distance in metric graphs.
\newblock {\em ArXiv e-prints}, 2024, arXiv:2411.09182 [math.MG].

\bibitem{AlimovTsarkov2019}
A.R.~Alimov and I.G.~Tsar'kov.
\newblock Chebyshev centres, Jung constants, and their applications.
\newblock {\em Russian Mathematical Surveys}, 74(5):775--849, 2019.

\bibitem{Amir1985}
D.~Amir.
\newblock On Jung's constant and related constants in normed linear spaces.
\newblock {\em Pacific Journal of Mathematics}, 118(1--15):359--371, 1985.

\bibitem{Berdyshev1998}
S.V.~Berdyshev.
\newblock The relative Jung constant in the space $l_n^\infty$.
\newblock {\em Tr.\ Inst.\ Mat.\ Mekh.\ \(Ekaterinburg\/\)}, 5:97--103, 1998.
\newblock (In Russian).

\bibitem{Bohnenblust}
F.~Bohnenblust.
\newblock Convex regions and projections in Minkowski spaces.
\newblock {\em Ann.\ Math.}, 39:301-308, 1938.

\bibitem{B75}
R.D.~Bourgin.
\newblock Approximate isometries on finite dimensional {B}anach spaces.
\newblock {\em Trans.\ Amer.\ Math.\ Soc.}, 207:309--328, 1975.

\bibitem{BBI}
D.~Burago, Yu.~Burago, and S.~Ivanov.
\newblock A Course in Metric Geometry.
\newblock {\em Graduate Studies in Mathematics}, 33, AMS, Providence, RI, 2001.

\bibitem{ChazalDeSilvaOudot2014}
F.~Chazal, V.~de~Silva, and S.~Oudot.
\newblock Persistence stability for geometric complexes.
\newblock {\em Geometriae Dedicata}, 174:193--214, 2014.

\bibitem{danzer1963helly}
L.~Danzer, B.~Gr{\"u}nbaum, and V.~Klee.
\newblock Helly's theorem and its relatives.
\newblock In {\em Proceedings of Symposia in Pure Mathematics}, pages 101--180.
  American Mathematical Society, 1963.

\bibitem{Dol'nikov}
V.L.~Dol'nikov. 
\newblock Jung constant in $\ell^n_1$.
\newblock {\em Math.\ Notes}, 42(4):787--791, 1987.

\bibitem{Garkavi1964}
A.L.~Garkavi.
\newblock On the Chebyshev center and convex hull of a set.
\newblock {\em Uspekhi Mat.\ Nauk}, 19(6(120)):139--145, 1964.
\newblock (In Russian).

\bibitem{GrigorevIvanovTuzhilin}
D.S.~Grigor'ev, A.O.~Ivanov, and A.A.~Tuzhilin.
\newblock Gromov--Hausdorff distance to simplexes.
\newblock {\em ArXiv e-prints}, 2019, arXiv:1906.09644v1 [math.MG].

\bibitem{Gulevich}
N.M.~Gulevich.
\newblock On measure of nonconvexity and {J}ung constant.
\newblock {\em J.\ Math.\ Sci.}, 81(2):2562--2566, 1996.

\bibitem{harrison2023quantitative}
M.~Harrison and R.A.~Jeffs.
\newblock Quantitative upper bounds on the {G}romov--{H}ausdorff distance between spheres.
\newblock {\em ArXiv e-prints}, 2023, arXiv:2309.11237 [math.MG].

\bibitem{Hatcher}
A.~Hatcher. 
\newblock  Algebraic topology.
\newblock {\em Cambridge University Press}, 2002.

\bibitem{Ilyukhin}
D.A.~Ilyukhin.
\newblock A problem of intersection of balls in normed space.
\newblock {\em ArXiv e-prints}, 2026, arXiv:2606.27583 [math.MG].

\bibitem{IvaMikhTuz} A.O.~Ivanov, I.N.~Mikhailov, and A.A.~Tuzhilin.
\newblock Gromov-Hausdorff Geometry of Metric Trees.
\newblock {\em ArXiv e-prints}, 2024, arXiv:2412.18888 [math.MG].

\bibitem{JiTuzhilin}
Y.~Ji and A.A.~Tuzhilin.
\newblock Gromov-Hausdorff Distance Between Segment and Circle.
\newblock {\em ArXiv e-prints}, 2021, arXiv:2101.05762 [math.MG].

\bibitem{jung1910kleinsten}
H.W.~Jung.
\newblock {\"U}ber die kleinste Kugel, die eine r\"aumliche Figur einschliesst.
\newblock {\em J.~Reine Angew.\ Math.}, 123:241--257, 1901.

\bibitem{Klee1960}
V.L.~Klee.
\newblock Circumspheres and inner products.
\newblock {\em Math.\ Scand.}, 8(2):363--370, 1960.

\bibitem{LimMemoliSmith}
S.~Lim, F.~M\'{e}moli, and Z.~Smith.
\newblock The Gromov--Hausdorff distance between spheres.
\newblock {\em Geometry \& Topology}, 2023, v.\ 27, N 9, pp.\ 3733--3800.

\bibitem{Martin}
S.R.~Martin.
\newblock Some novel constructions of Gromov-Hausdorff-optimal correspondences between spheres.
\newblock {\em ArXiv e-prints}, 2025, arXiv:2409.02248v2 [math.MG].

\bibitem{MikhTuz}
I.N.~Mikhailov, A.A.~Tuzhilin.
\newblock When the Gromov--Hausdorff distance between finite-dimensional space and its subset is finite?
\newblock {\em Communications in Mathematical Research}, 2025, v.\ 41, iss. 1, pp.\, 1--8.

\bibitem{Routledge} N.~Routledge.
\newblock A result in Hilbert space.
\newblock {\em Quart.\ J.\ Math.}, 3:12--18, 1952.

\bibitem{Stone} A.H.~Stone.
\newblock Paracompactness and product spaces.
\newblock {\em Bulletin of the American Mathematical Society}, 1948, v.\ 54, N 10, pp.\ 977--982.

\end{thebibliography}
\end{document}